\documentclass[12pt]{amsart}

\usepackage{tikz}

\usepackage{booktabs}
\usepackage{multicol}

\usepackage{amsthm}

\usepackage{dsfont}

\usepackage{lscape,amssymb, amsmath,verbatim,graphicx,color}
\usepackage{epsfig,subfigure,float}
\usepackage{graphicx}
\usepackage{epsfig}
\usepackage{siunitx}
\usepackage{multirow}
\usepackage{adjustbox}
\usepackage[margin=1in]{geometry}
\newtheorem{assumption}{Assumption}[section]

\newcommand{\mN}{\mathcal N}
\newcommand{\Tx}{\lfloor Tx \rfloor}
\newcommand{\intt}{\int\hspace{-.2cm}\int}

\newtheorem{cor}{Corollary}[section]

\newtheorem{theorem}{Theorem}[section]
\newtheorem{lemma}{Lemma}[section]

\newtheorem{defi}{Definition}

\allowdisplaybreaks
\def\beq{\begin{equation}}
\def\eeq{\end{equation}}

\graphicspath{{plots/}}

\numberwithin{table}{section}
\numberwithin{figure}{section}

\allowdisplaybreaks
\def\beq{\begin{equation}}
\def\eeq{\end{equation}}

\numberwithin{equation}{section}
\numberwithin{theorem}{section}

\graphicspath{{plots/}}

\usepackage{setspace}

\title[]{\bf Inference for the cross-covariance operator of stationary functional time series}

\author{Gregory Rice}
\address{ Department of Statistics and Actuarial Science, University of Waterloo, Waterloo, ON, Canada}
\author{Marco Shum}
\address{ Department of Statistics and Actuarial Science, University of Waterloo, Waterloo, ON, Canada}

\date{}
\begin{document}

%\author {Gregory Rice}
%%$^{(1)}$}
%\address{Gregory Rice, Department of Mathematics, University of Utah, Salt Lake City, UT 84112--0090 USA }
%\\ email: horvath@math.utah.edu
%}

%\subjclass{62M10, 62F05, 60F17}
%\keywords{panel data, change in the mean, time series, weak convergence, CUSUM process}

%\subjclass{Primary: 62G10; Secondary:  62M10, 62H15}
%\keywords{functional data analysis, time series, change point analysis}
\maketitle

%\thanks{ Research supported by NSF grant DMS 1305858   }
%\doublespacing

%\vspace{-2cm}

\begin{abstract}

When considering two or more time series of functions or curves, for instance those derived from densely observed intraday stock price data of several companies, the empirical cross-covariance operator is of fundamental importance due to its role in functional lagged regression and exploratory data analysis. Despite its relevance, statistical procedures for measuring the significance of such estimators are undeveloped. We present methodology based on a functional central limit theorem for conducting statistical inference for the cross-covariance operator estimated between two stationary, weakly dependent, functional time series. Specifically, we consider testing the null hypothesis that two series possess a specified cross-covariance structure at a given lag. Since this test assumes that the series are jointly stationary, we also develop a change-point detection procedure to validate this assumption, which is of independent interest. The most imposing technical hurdle in implementing the proposed tests involves estimating the spectrum of a high dimensional spectral density operator at frequency zero. We propose a simple dimension reduction procedure based on functional PCA to achieve this, which is shown to perform well in a small simulation study. We illustrate the proposed methodology with an application to densely observed intraday price data of stocks listed on the NYSE.

\end{abstract}

\section{Introduction}

 %as in functional magnetic resonance imaging, in which functions describing blood flow in the brain are computed over time, see \cite{aston:kirch:2012}, or in genetics when the shapes of DNA minicircles are continuously recorded, see \cite{panaretos:tavakoli:2015}.

Functional time series analysis (FTSA) has grown substantially in the last decade and a half in order to provide methodology for functional data objects that are obtained sequentially over time. Perhaps the most typical way such data arises is when dense records of continuous time processes are segmented into collections of curves in some natural way. For example, high frequency records of pollution levels may be segmented to form daily pollution curves, or tick-by-tick asset price data may be used to construct daily intraday price or return curves; see \cite{aue:norinho:hormann:2014} and  \cite{kargin:onatski:2008}. Other examples include sequentially observed curves that describe physical phenomena, as arise in functional magnetic resonance imaging or in the observation of DNA minicircles; see \cite{aston:kirch:2012} and \cite{panaretos:tavakoli:2015}. We refer the reader to \cite{ramsay:silverman:2005} and \cite{ferraty:vieu:2006}  for overviews of the field of functional data analysis, and to \cite{bosq:2000} and \cite{hormann:kokoszka:2012} for an introduction to FTSA.

Most of the developments in these two fields focus on analyzing functional data obtained from a single source, e.g. intraday price curves derived from a single asset, or in comparing functional data from several independent populations. To give a few examples that are related to this work, \cite{panaretos:2010}, \cite{paparoditis:spatinas:2015}, and \cite{pigoli:2014} develop methods for performing inference for the covariance operator of functional data. Using a self normalization approach, a two sample test for the second order structure with functional time series data that allows for some dependence across the populations is developed in \cite{zhang:shao:2015}.

However, in many situations of interest,  functional data are obtained simultaneously from two or more sources, e.g. intraday price curves derived from several assets. In such cases one often wishes to quantify the potentially complex dependence relationships between such curves, and one way of achieving this is through the empirical cross-covariance operator. Although the notion of the cross-covariance operator between random elements in a Hilbert space was put forward over forty years ago in \cite{baker:1973}, statistical methodology for estimating and performing further inference for the cross-covariance structure between collections of curves seems quite new.

Measuring the cross-covariance between collections of curves has received some attention in the context of multivariate longitudinal and functional data. Under the assumption that the given data is a simple random sample of multivariate longitudinal data, \cite{dubin:muller:2005}, \cite{serban:staicu:carroll:2013}, and \cite{zhou:huang:carroll:2008} develop measures of cross-covariance between longitudinal data sources, including measures based on canonical correlation analysis and principal component analysis. Principal component analysis of multivariate functional data is also studied in \cite{chiou:chen:yang:2014} and \cite{chiou:muller:2013}, and in \cite{petersen:muller:2016} an analog of the covariance matrix for multivariate functional data using Fr\'echet integration is defined. In each of these cases, the potential effect of temporal dependence among the functional units is not considered.

\begin{comment}
Functional two sample and  analysis of variance tests for the mean are considered in \cite{staicu:2015} and \cite{cuevas:febrero:freiman:2004}, among many others.
\end{comment}

In the context of bivariate functional time series, the cross-covariance operator and its lagged versions are arguably of greater importance. Methods for lagged functional time series regression, which have recently been put forward in \cite{hormann:kidzinski:kokoszka:2015} and \cite{pham:panaretos:2017}, are naturally based on the Fourier transform of the lagged cross-covariance operators. Moreover, the initial exploratory analysis of any such series would typically begin by considering the sequence of lagged cross-covariance operators to try to gain insight into the relationship between the series.

Despite the apparent utility of the cross-covariance operator of functional time series, statistical inference for it has not yet been considered, to the best of our knowledge. In Chapter 4 of the seminal work of Bosq \cite{bosq:2000}, a central limit theorem is given for the covariance and auto-covariance operators of functional time series that may be represented as linear processes. A portmanteau-type test for independence of two functional time series is developed in \cite{h:r:2014ti} based on the norms of cross-covariance operators at long lags, but their test assumes under the null hypothesis that the individual series are independent, and is hence not suitable for quantifying the significance of estimates of the cross-covariance between curves.

Additionally, when the data are obtained as bivariate functional time series, it is of special interest to know if the cross-covariance structure changes during the observation period. Being able to test for such a feature 1) helps validate the assumption of joint stationarity needed to apply inferential procedures for the cross-covariance operator, 2) is of use for determining if the regression function changes in the functional lagged regression problem, and 3) may be of independent interest since the presence and location of such a change point may signify an important event. This problem has also not been addressed, although several authors have considered analogous problems in the context of finite dimensional time series; we refer to \cite{dette:wu:zhou:2015}, \cite{weid1:2012}, \cite{wied2:2012}, and \cite{aue:hormann:horvath:reimherr:2009}. Change point analysis for the mean of functional time series has been considered recently in \cite{wendler1:2016} and \cite{wendler2:2016}.

In this paper, we consider two hypothesis testing problems: 1) tests for a specified cross-covariance structure between two functional time series, e.g. that the two series are uncorrelated at a given lag, and 2) change point tests for the covariance structure within a given sample. Two varieties of test statistics are proposed in each of these settings that are based on either the standard $L^2$ distance or dimension reduction based methods using a suitable principal component basis. These two classes of statistics possess complimentary advantages, which we detail by means of a theoretical result. The asymptotic properties of each test statistic are established assuming a general weak dependence condition similar to the one introduced in \cite{hormann:kokoszka:2010}, which includes nonlinear time series, as well as the majority of functional time series models studied to date, under mild regularity conditions.

This methodology is primarily motivated by a basic observation of Brillinger \cite{brillinger:1975} that inference for the covariance and/or cross-covariance of time series can be made by performing inference for the mean of a suitably constructed series. This idea was utilized in \cite{himdi:roy:duchesne:2003} in order to conduct non-parametric inference for the cross-covariance matrix of finite dimensional time series.

The rest of the paper is organized as follows. Section \ref{main} contains the main assumptions and notation of the paper, as well as asymptotic results for cross-covariance function estimators under these assumptions. The results developed in Section \ref{main} are utilized to develop hypothesis tests in Section \ref{i-1}. Some details about the practical implementation of the methods developed in Section \ref{i-1}  are provided in Section \ref{imp}, which include methods for overcoming the technical challenge of estimating the eigen-elements of a high-dimensional spectral density operator at frequency zero that arise in the limiting distribution of the test statistics. These testing and estimation procedures were studied by means of Monte Carlo simulation, the results of which we present in Section \ref{simul}. We illustrate our methodology with an application to cumulative intraday return curves derived from the 1-minute resolution price of Microsoft and Exxon Mobile stock listed on the New York stock exchange from the year 2001 in Section \ref{app}. All technical derivations and proofs are provided in the appendices following these sections.

\section{Asymptotic properties of cross-covariance function estimates}\label{main}

Before we proceed, we introduce a bit of notation. Let $\langle \cdot, \cdot \rangle_d$ denote the standard inner product on the space $L^2[0,1]^d$ of real valued square integrable functions defined on $[0,1]^d$, and let $\| \cdot \|_d =\langle \cdot, \cdot \rangle_d^{1/2}$. We write $f$ for the function $f(t)$ when it does not cause confusion. We use the notation $\int$ to denote $\int_0^1$.

In this section, we suppose that $ \{ (X_i(t),Y_i(s)),\; t,s \in[0,1] \}_{i\in {\mathbb Z}}$ is a jointly stationary sequence of real valued stochastic processes whose sample paths are in $L^2[0,1]$ from which we have observed a sample of length $T$, $\{(X_1(t),Y_1(s)),...,(X_T(t),Y_T(s))\}$. For instance, $(X_i(t),Y_i(s))$ could be used to denote the price of stock $X$ and stock $Y$ on day $i$ at intraday times $t$ and $s$ normalized to the unit interval. We let $\mu_X(t)=EX_0(t)$, and $\mu_Y(s)= EY_0(s)$, and define

$$
C_{XY}(t,s) = \mbox{cov}(X_0(t),Y_0(s)) = E[(X_0(t)-\mu_X(t))(Y_0(s)-\mu_Y(s))],
$$

to be the cross-covariance function (or kernel) between $\{X_i\}$ and $\{Y_i\}$ at lag zero. $C_{XY}$ defines the cross-covariance operator $c_{XY}:L^2[0,1]\to L^2[0,1]$ via

$$
c_{XY}(f)(t) = \int C_{XY}(t,s)f(s)ds.
$$

This relationship implies that we may conduct inference for $c_{XY}$ by conducting inference for the function $C_{XY}$. Based on the sample, we may estimate $C_{XY}$ with

$$
\hat{C}_{XY}(t,s,x) = \frac{1}{T} \sum_{i=1}^{\lfloor Tx \rfloor} (X_i(t) - \bar{X}(t))(Y_i(s)- \bar{Y}(s)),
$$
which denotes the partial sample estimate of $xC_{XY}$ based on the first $x$ proportion of the sample, where

$$\bar{X}(t) = \frac{1}{T} \sum_{i=1}^T X_i(t), \mbox{ and } \bar{Y}(t) = \frac{1}{T} \sum_{i=1}^T Y_i(t).$$

Under mild regularity conditions on the process $\{(X_i(t),Y_i(s))\}$, which are implied by the Assumption 2.1 below, $\hat{C}_{XY}(t,s,1)$ is a consistent estimator of $C_{XY}(t,s)$ in $L^2[0,1]^2$. The motivation for considering now the partial sample estimates of $C_{XY}$ is due to our application to change point testing for the cross-covariance developed below. We note that when interested in studying the cross-covariance of the series at some fixed lag that is different from zero, say $\ell>0 $, one can apply all of the below methods to the sample $\{(X_{\ell+1}(t),Y_1(s)),...,(X_T(t),Y_{T-\ell}(s))\}$ of length $T-\ell$ with only superficial changes to the proofs and notation, and so we otherwise make no special mention of this case.

In order to derive the asymptotic properties of $\hat{C}_{XY}$, we make use of the following assumption that imposes stationarity, weak dependence, and moment conditions on the functional time series.

\begin{assumption}\label{edep}

(a) There exists a measurable function $g_{XY}\colon S^\infty\to L^2[0,1]\times L^2[0,1]$, where $S$ is a measurable space, and a sequence of independent and identically distributed (iid) innovations $\{\epsilon_i,\colon i\in\mathbb{Z}\}$ taking values in $S$ such that $(X_i,Y_i)=g_{XY}(\epsilon_i,\epsilon_{i-1},\ldots)$.

(b) For all $m\ge 1$, the $m$-dependent sequence $(X_{i,m},Y_{i,m})=g(\epsilon_i,\ldots,\epsilon_{i-m+1},\epsilon^*_{i-m,m},\epsilon^*_{i-m-1,m},\ldots)$ with $\epsilon^*_{i,m}$ being independent copies of $\epsilon_{i,0}$, and $\{\varepsilon_{i,m}^*\colon i\in\mathbb{Z}\}$ independent of $\{\epsilon_i\colon i\in\mathbb{Z}\}$, satisfies for some $p > 4$,

$$
\big(E\|X_i-X_{i,m}\|^p\big)^{1/p}=O(m^{-\alpha}),\mbox{ and } \big(E\|Y_i-Y_{i,m}\|^p\big)^{1/p}=O(m^{-\beta})
$$

where $\alpha,\beta >1$.

\end{assumption}

Assumption \ref{edep}(a) implies that $\{(X_i,Y_i)\}$ is a jointly stationary sequence of Bernoulli shifts in $L^2[0,1]\times L^2[0,1]$ that is driven by an underlying iid innovation sequence. The space of functional time series models contained within this class is quite large, including the functional ARMA and GARCH processes; see \cite{bosq:2000} and \cite{aue:horvath:pellatt:2016}. Condition (b) defines a type of $L^p$-$m$-approximability condition along the lines of \cite{hormann:kokoszka:2010}, which can often be easily verified when a time series model for the observations is given. The rate condition on the decay of these coefficients, which is somewhat stronger than the main condition studied in \cite{hormann:kokoszka:2010}, is used in order to show that certain smoothed periodogram type spectral density operator estimates defined below are consistent. These assumptions could be replaced by mixing conditions and functional versions of cummulant summability conditions as presented in \cite{panaretos:tavakoli:2012} and \cite{zhang:2016}, which are more comparable to the assumptions in \cite{himdi:roy:duchesne:2003}.

\begin{theorem}\label{th-1} Under Assumption \ref{edep}, there exists a sequence of Gaussian processes, $\{\Gamma_T(t,s,x),\; t,s,x\in[0,1]\}_{T\in {\mathbb N}}$, defined on the same probability space as $\{(X_i,Y_i)\},$ that satisfy

\begin{align}\label{th-1-eq}
\sup_{0\le x \le 1} \intt \left( \sqrt{T}\left(\hat{C}_{XY}(t,s,x)-\frac{\lfloor Tx \rfloor }{T} C_{XY}(t,s)\right) -\Gamma_T(t,s,x)\right)^2dtds = o_P(1),
\end{align}
where $E\Gamma_T(t,s,x)=0$, and

$$
\mbox{cov}(\Gamma_T(t,s,x),\Gamma_T(u,v,y))=\min(x,y) D(t,s,u,v),
$$
where $D(t,s,u,v)$ is the long run covariance function of the sequence $\{(X_i(t)-\mu_X(t))(Y_i(s)-\mu_Y(s))\}$, namely

$$
D(t,s,u,v) = \sum_{\ell=-\infty}^\infty \mbox{cov}((X_0(t)-\mu_X(t))(Y_0(s)-\mu_Y(s)),(X_\ell(u)-\mu_X(u))(Y_\ell(v)-\mu_Y(v)).
$$

\end{theorem}

Theorem \ref{th-1} provides a Skorokhod-Dudley-Wichura type characterization of an invariance principle for $\hat{C}_{XY}$ that can be utilized to establish the asymptotic properties of continuous functionals of $\hat{C}_{XY}$, and we consider several such statistics in Section 3 below in order to carry out hypothesis testing for $C_{XY}$. The proof of Theorem \ref{th-1} is given in Appendix \ref{proofs}. The function $D$ describes the asymptotic covariance function of $\sqrt{T}\hat{C}_{XY}(t,s,1)$. $D$ naturally defines a Hilbert-Schmidt integral operator, $d:L^2[0,1]^2 \to L^2[0,1]^2$, given by

%This result follows by modifying the main results in \cite{berkes:horvath:rice:2013} and \cite{jirak:2013} to this setting.
$$
d(f)(t,s) = \intt D(t,s,u,v) f(u,v)dudv,
$$
which further defines an orthonormal basis of eigenfunctions $\varphi_i$ in $L^2[0,1]^2$, and a nonnegative sequence of eigenvalues $\lambda_1 \ge \lambda_2 \ge...\ge 0$ satisfying

\begin{align}\label{eigeneq}
d(\varphi_i)(t,s) = \lambda_i \varphi_i(t,s).
\end{align}
 We define these quantities here as they appear in the limiting distributions and definitions for the test statistics considered below.

\section{Inference for the cross-covariance function}\label{i-1}

 Theorem \ref{th-1} points to some asymptotically validated methods to measure the significance of estimates of $C_{XY}$. For instance, we may wish to test based on the estimate $\hat{C}_{XY}(\cdot,\cdot,1)$

\begin{align*}%\label{h0-1}
H_{0,1}: \; C_{XY}=C_0 \; \mbox{ versus } H_{A,1}: \;\; C_{XY}\ne C_0,
\end{align*}
where equality is understood in the $L^2[0,1]^2$ sense, and $C_0$ is a given function of interest. This null function might be determined from historical data, or taken to be zero in order to test for zero cross-covariance between $X_i$ and $Y_i$ at a given lag. Since the hypothesis $H_{0,1}$ is well posed only when the sequence $\{(X_i,Y_i)\}$ is, at least weakly, jointly stationary, it is also of interest to determine whether or not this assumption is valid. We frame this as a second hypothesis test of the time homogeneity of the cross-covariance against the ``at most one" change point in the cross-covariance alternative:

$$
H_{0,2}:\;\; C_{XY}^{(1)}= C_{XY}^{(2)}= \cdots = C_{XY}^{(T)}, \mbox{ versus }
$$
$$
H_{A,2}: \;\; C_1= C_{XY}^{(1)}=  \cdots = C_{XY}^{(k^*)} \ne C_{XY}^{(k^*+1)} = \cdots =C_{XY}^{(T)}=C_2, \mbox{ for some } k^*=\lfloor T \theta \rfloor,\;\; \theta \in(0,1),
$$

where $C_{XY}^{(i)}(t,s)=\mbox{cov}(X_i(t),Y_i(s))$. We proceed by developing test statistics for each of these hypotheses. In order to test $H_{0,1}$, we first consider a statistic based on the normalized $L^2$ distance of $\hat{C}_{XY}$ to $C_0$:

$$
F_T = T \| \hat{C}_{XY}(\cdot,\cdot,1)-C_0 \|_2^2.
$$

\begin{cor}\label{cor-1} Under Assumption \ref{edep} and $H_{0,1}$,

\begin{align}\label{cor1eq}
F_T \stackrel{D}{\longrightarrow} \sum_{i=1}^\infty \lambda_i \mN_i^2,\mbox{ as } T \to \infty,
\end{align}
where $\{\lambda_i, \; i\ge1 \}$ are defined in \eqref{eigeneq}, and $\{\mN_i$, $i\ge 1\}$ are independent standard normal random variables.
\end{cor}

Corollary \ref{cor-1} shows that a test of asymptotic size $\alpha$ of $H_{0,1}$ may be obtained by comparing $F_T$ to the $1-\alpha$ quantile of the limit distribution given in \eqref{cor1eq}, which depends on the unknown eigenvalues of the operator $d$. In order to estimate these quantiles, one can estimate a suitably large number of eigenvalues $\lambda_i$ using an estimate of $d$, and then continue by using these estimates to approximate the limiting distribution via Monte Carlo simulation. The details of this implementation are discussed in Section \ref{imp} below, including how to obtain consistent estimates of a finite number of the $\lambda_i's$.

The fact that the limiting distribution of $F_T$ is nonpivotal though encourages one to consider alternate test statistics based on projecting $\hat{C}_{XY}$ into finite dimensional subspaces of $L^2[0,1]^2$. A natural choice of the finite dimensional space to choose is the one spanned by the eigenbasis generated by $d$. In fact, it is a fairly straightforward calculation to show that for any positive integer $p$, under Assumption \ref{edep}, the inner products

$$
\langle \sqrt{T}(\hat{C}_{XY}(\cdot,\cdot,1)-C_{XY}), \varphi_i \rangle_2, \; 1 \le i \le p
$$
are asymptotically independent, and hence projecting  into the directions of $\varphi_i$ has the effect of partitioning the centered estimator $\hat{C}_{XY}$ into approximately mean zero and independent components. This is similar to the motivation provided for dynamical principal component analysis of Brillinger \cite{brillinger:1975}, which has been studied in the context of functional time series data in \cite{hormann:kidzinski:hallin:2015}, and \cite{panaretos:tavakoli:2012}.

In this direction, let

$$
F_{T,p}=\sum_{i=1}^p \frac{\langle \sqrt{T}(\hat{C}_{XY}(\cdot,\cdot,1)-C_0), \hat{\varphi}_i \rangle^2_2}{\hat{\lambda}_i},
$$

where $\hat{\varphi}_i$, and $\hat{\lambda}_i$, $1\le i \le p$ are consistent estimates of $\varphi_i$ and $\lambda_i$, $1 \le i \le p$, i.e. they satisfy

\begin{align}\label{spec-est}
\max_{1 \le i \le p} \| \hat{\varphi}_i - \hat{c}_i\varphi_i \|_2= o_P(1), \mbox{ and }\max_{1 \le i \le p}|\hat{\lambda}_i - \lambda_i|=o_P(1),
\end{align}
where $\hat{c}_i = sign(\langle \varphi_i, \hat{\varphi}_i\rangle )$. We discuss in Section $\ref{imp}$ how to obtain consistent estimates of $\lambda_i$ and $\varphi_i(t,s)$.

\begin{cor}\label{cor-2} Under Assumption \ref{edep} and $H_{0,1}$,

\begin{align*}
F_{T,p} \stackrel{D}{\longrightarrow} \chi^2(p),\mbox{ as } T \to \infty,
\end{align*}
where $\chi^2(p)$ denotes a chi-squared random variable with $p$ degrees of freedom.
\end{cor}

We now turn to the consistency and power properties of $F_T$ and $F_{T,p}$. The following result shows that both statistics diverge at rate $T$ under $H_{A,1}$, so long as the difference $C_{XY}-C_0$ is not orthogonal to the first $p$ elements of the  principal component basis $\{\varphi_i,\; i\ge 1\}$.

\begin{theorem}\label{alt-1-th} Under Assumption \ref{edep}, and $H_{A,1}$,

$$
\frac{F_T}{T} \stackrel{P}{\longrightarrow} \|C_{XY}-C_0\|^2, \mbox{ and } \; \frac{F_{T,p}}{T} \stackrel{P}{\longrightarrow} \sum_{\ell=1}^p \frac{\langle C_{XY}-C_0, \varphi_i \rangle^2_2}{\lambda_i}.
$$
Moreover, if $C_0=C_{0,T}$ satisfies $\| \sqrt{T}(C_{XY}-C_0)-C_A\|_2\to 0$ as $T\to \infty$ for some element $C_A$ of $L^2[0,1]^2$, then

\begin{align}\label{contig-1}
F_T \stackrel{D}{\longrightarrow} \|C_A\|_2^2 + 2 \sum_{i=1}^\infty \lambda^{1/2}_i \langle C_A, \varphi_i \rangle \mN_i + \sum_{i=1}^\infty \lambda_i \mN_i^2,
\end{align}
and

\begin{align*}
F_{T,p} \stackrel{D}{\longrightarrow} \sum_{i=1}^p \frac{\langle C_A , \varphi_i \rangle^2_2}{\lambda_i} + 2 \sum_{i=1}^p \frac{\langle C_A , \varphi_i \rangle_2 \mN_i}{\lambda^{1/2}_i} + \sum_{i=1}^p \mN_i^2 .
\end{align*}

\end{theorem}

This ``local-alternative" result provides insight into the complimentary strengths and weaknesses of both the norm based and dimension reduction based test statistics. Evidently if $C_{XY}-C_0$ is orthogonal to the first $p$ principal components of $d$, then the test statistic $F_{T,p}$ is not expected to have more than trivial power. Additionally in this case, the norm based test statistic $F_T$ is expected to have improved power over other alternatives in which $C_{XY}-C_0$ is of the same magnitude, since the second term on the right hand side \eqref{contig-1}, which has mean zero, will have a smaller variance in this case. Conversely, if $C_{XY}-C_0$ is contained in the subspace spanned by the first $p$ principal components of $d$, then $F_{T,p}$ is expected to be more powerful, since then this statistic effectively defines a uniformly most powerful test of $H_{0,1}$, assuming the data is Gaussian, in the $p$-dimensional subspace spanned by $\varphi_1,...,\varphi_p$.

In order to test $H_{0,2}$ versus $H_{A,2}$, we define analogously to $F_T$ and $F_{T,p}$,

$$
Z_T =  T \sup_{0 \le x \le 1}  \| \hat{C}_{XY}(\cdot,\cdot,x)- x\hat{C}_{XY}(\cdot,\cdot,1) \|_2^2.
$$
and
$$
Z_{T,p} =  T \sup_{0 \le x \le 1}  \sum_{i=1}^p \frac{ \langle \hat{C}_{XY}(\cdot,\cdot,x)- x\hat{C}_{XY}(\cdot,\cdot,1),\hat{\varphi}_i \rangle_2^2}{\hat{\lambda}_i}.
$$

$Z_T$ and $Z_{T,p}$ are each maximally selected CUSUM type statistics based on comparing the partial sample estimates of $C_{XY}$ to the estimator from the whole sample. The following corollaries of Theorem \ref{th-1} quantify the large sample behavior of these statistics under $H_{0,2}$.

\begin{cor}\label{cor-3} Under Assumption \ref{edep} and $H_{0,2}$,
\begin{align*}
Z_T \stackrel{D}{\longrightarrow} \sup_{0 \le x \le 1} \sum_{i=1}^\infty \lambda_i B^2_i(x),\mbox{ as } T \to \infty,
\end{align*}

where $\{\lambda_i, \; i\ge1 \}$ are defined in \eqref{eigeneq}, and $\{B_i(x)$, $i\ge 1,\;x\in[0,1]\}$ are iid standard Brownian bridges on $[0,1]$. If in addition \eqref{spec-est} holds, then

\begin{align*}%\label{kief}
Z_{T,p} \stackrel{D}{\longrightarrow} \sup_{0 \le x \le 1} \sum_{i=1}^p  B^2_i(x),\mbox{ as } T \to \infty,
\end{align*}

\end{cor}

It follows then that a test of $H_{0,2}$ with asymptotic level $\alpha$ is obtained by rejecting if $Z_T$ or $Z_{T,p}$ are larger than the $1-\alpha$ quantiles of their limiting distributions detailed in Corollary \ref{cor-3}. These limiting distributions may again be approximated using Monte Carlo simulation. The statistics $Z_T$ and $Z_{T,p}$ diverge under $H_{A,2}$ in conjunction with some mild ergodicity assumptions, which we explain in Section \ref{alt-sec} in the appendix.

\section{Implementation } \label{imp}

Implementing the testing procedures outlined above requires the estimation of the eigenvalues, and, in case of the dimension reduction based test statistics $F_{T,p}$ and $Z_{T,p}$, the eigenfunctions of $d$. We first develop methodology for estimating $d$ and its spectrum, and then describe some numerical methods for carrying out this estimation.

\subsection{Estimation of $D$, $d$, $\varphi_i$ and $\lambda_i$}

As $d$ is simply a scalar multiple of the spectral density operator at frequency zero of the stationary sequence $\{(X_i(t)-\mu_X(t))(Y_i(s)-\mu_Y(s))\}$ in $L^2[0,1]^2$, as defined in \cite{panaretos:tavakoli:2012}, it may be naturally estimated with a smoothed periodogram type estimator. Let

\begin{align}\label{est-1}
\hat{D}_T(t,s,u,v)=\sum_{\ell=-\infty}^{\infty}W_{b} \left( \frac{\ell}{h} \right) \hat{\gamma}_\ell(t,s,u,v),
\end{align}
where $h$ is a bandwidth parameter satisfying,

\begin{align}\label{h-cond}
h=h(T) \to \infty,\; \frac{h}{T^{1/2}} \to 0 \mbox{ as } T \to \infty,
\end{align}

and, with $\bar{X}_j(t) = X_j(t) -\bar{X}(t)$ and $\bar{Y}_j(s)$ similarly defined,
\begin{align*}
   \hat{\gamma}_\ell(t,s,u,v)=\left\{
     \begin{array}{lr}
      \displaystyle \frac{1}{T}\sum_{j=1}^{T-\ell}\left( \bar{X}_j(t)\bar{Y}_j(s)-\hat{C}_{XY}(t,s,1)\right)\left( \bar{X}_{j+\ell}(u)\bar{Y}_{j+\ell}(v)-\hat{C}_{XY}(u,v,1)\right),\quad &\ell \ge 0
      \vspace{.3cm} \\
     \displaystyle \frac{1}{T}\sum_{j=1-\ell}^{T}\left( \bar{X}_j(t)\bar{Y}_j(s)-\hat{C}_{XY}(t,s,1)\right)\left( \bar{X}_{j+\ell}(u)\bar{Y}_{j+\ell}(v)-\hat{C}_{XY}(u,v,1)\right),\quad &\ell < 0.
     \end{array}
   \right.
\end{align*}

We take the function $W_b$ to be a symmetric and continuous weight function with bounded support of order $b$; see Chapter 7 of \cite{priestley:1981}. $\hat{D}_T$ then defines an estimator of $d$ by

$$
\hat{d}_T(f)(t,s) =  \intt \hat{D}_T(t,s,u,v) f(u,v)dudv,
$$
which further defines estimates of the eigenvalues and eigenfunctions of $d$ satisfying

\begin{align}\label{eig-est}
 \hat{d}_T(\hat{\varphi}_i)(t,s) =  \hat{\lambda}_i \hat{\varphi}_i(t,s).
\end{align}

In the simulations and application below, we take the weight function $W_b$ to be the simple Bartlett weight function, $W_1(x)= (1-|x|)\mathds{1}(|x|<1)$, which is of order 1, and $h = \lceil N^{1/5} \rceil$. Under Assumption \ref{edep}, we have the following result:

\begin{theorem}\label{th-est} Under Assumption \ref{edep}, $\| D - \hat{D}_T \|_4 =o_P(1).$
\end{theorem}

In order to obtain consistent estimates of the first $p$ elements of the spectrum of $d$, we assume for the sake of simplicity that the first $p$ eigenvalues of $d$ are distinct, which implies that the corresponding eigenspaces of the first $p$ eigenvalues are one dimensional.
\begin{assumption}\label{eigen-as}
We assume that there exists an integer $p \ge 1$ satisfying that
$$\lambda_1 > \cdots > \lambda_p > \lambda_{p+1} \ge 0$$
where $\{ \lambda_i,\; i \ge 1\}$ are defined in \eqref{eigeneq}.
\end{assumption}

Assumption \ref{eigen-as} could be relaxed by utilizing some of the ideas presented in \cite{reimherr:2015}, but we do not pursue adapting those here. Under this assumption, the following result is implied by Theorem \ref{th-est} and the results in Section 6.1 of \cite{gohberg:1990}.

\begin{cor}
Under Assumption \ref{edep} and \ref{eigen-as}, \eqref{spec-est} holds.
\end{cor}

\subsection{Numerical implementation  }\label{num-imp}

Although so far we have presented results as if the functional data at hand were observed on their entire domains, in practice the data will consist of only discrete observations of the underlying functions. Let  $X_i(t_j)$ and $Y_i(t_j)$, $1\le j \le R$, denote the observed values of the functions $X_i(t)$ and $Y_i(t)$, observed at the common points $\{t_1,\hdots, t_R\}$. We assume here that each functional observation $X_i(t)$ and $Y_i(s)$ are observed at common points in their domains, as this matches our simulations and data example below, although this could be easily relaxed.

It is straightforward to estimate the test statistics $F_T$, $F_{T,p}$, $Z_T$, and $Z_{T,p}$ from the discrete data and simple Riemann sum approximations to the inner products and norms, so long as the eigenvalue and eigenfunction estimates $\hat{\lambda}_i$ and $\hat{\varphi}_i$ are given. However it is less clear how to estimate the eigen-elements satisfying \eqref{eig-est} from the discrete data. One could in principle estimate the eigenvalues and eigenfunctions  $(\lambda_i,\varphi_i(t,s)),\;\; 1\le i \le p$ of $d$ by calculating the spectrum of the 4-way tensor $\hat{D}_T(t_i,t_j,t_k,t_\ell)$, $1\le i,j,k,\ell \le R$ of dimension $R^4$, and employing linear interpolation to complete the eigenfunction estimates, but this becomes computationally infeasible for even moderate values of $R$ due to the sheer dimension of the tensor, and the fact that this tensor is typically {\it dense}.

What we propose instead is a dimension reduction based approach involving the functional principal components (fPC's) of the individual series $\{X_i\}$ and $\{Y_i\}$. Let
\[ \widehat{c}_X := \left[ \frac{1}{T} \sum_{\ell=1}^T (X_\ell(t_i) - \bar{X}(t_i))(X_\ell(t_j) - \bar{X}(t_j)):\ 1 \leq i,j \leq R \right] \]
be the sample covariance matrix of the discretized observations of the $X$ sample, and define $\widehat{c}_Y$ similarly. Calculating the spectrum of  $\widehat{c}_X$ and $\widehat{c}_Y$ yields eigenvalues $\{\nu_{X,1}, \hdots, \nu_{X,R}\}$, $\{\nu_{Y,1}, \hdots, \nu_{Y,R}\}$ and eigenvectors, $\{\theta_{X,1}, \hdots, \theta_{X,R}\}$, $\{\theta_{Y,1}, \hdots, \theta_{Y,R}\}$, the latter of which, when multiplied by $\sqrt{R}$, yield discrete approximations to the fPC's of the sequences $\{X_i\}$ and $\{Y_i\}$, $\theta_{X,i}(t_j)$, $\theta_{Y,i}(t_j)$, $1\le i,j \le R$, respectively.

We then use the product basis of $L^2[0,1]^4$ generated by these functions to reduce the dimension of $\hat{D}_T$. The projections of $\hat{D}_T$ onto the product basis of the first $q$ elements each of $\theta_{X,i}(t)$, $\theta_{Y,i}(t)$ may be stored in a 4-way tensor with $q^4$ elements, $\mathcal{M}$, via
\[ \mathcal{M}_{ijkr} = \int \cdots \int \hat{D}_T(t,s,u,v) \theta_{X, i}(t) \theta_{Y, j}(s ) \theta_{X, k}(u) \theta_{Y, r}(v )dtdsdudv.\]
Each of these elements can be estimated with a simple Riemann sum approximation. The cutoff $q$ must be selected by the user. We suggest using the total variance explained approach for this: we take $q$ so that $\theta_{X,i}(t)$ and $\theta_{Y,i}(t)$, $1\le i \le q$, explains at least $v\%$ of the variation in each series, where $v$ is close to but strictly smaller than 1.

The eigenvalues $\{ \hat{\lambda}_i, 1 \le i \le q^2\}$ and eigen-arrays $\{\hat{{\bf \Phi}}_i \in \mathbb{R}^{q \times q}, 1\le i \le q^2 \}$ of $\mathcal{M}$ satisfy

\[  \mathcal{M} \hat{{\bf \Phi}}_\ell = \hat{\lambda}_\ell \hat{{\bf \Phi}}_\ell,\;\;1\le \ell \le q^2,\;\mbox{where}\;\; \mathcal{M} \hat{{\bf \Phi}}_\ell[i,j] = \sum_{k,r = 1}^q \mathcal{M}_{ijkr}
\hat{{\bf \Phi}}_\ell[k,r].  \].

This eigenvalue problem may be solved numerically by solving for the eigenvalues/vectors of a $q^2$ by $q^2$ square matrix that is ``tiled" with the cross-sections of $\mathcal{M}$, as is implemented in \texttt{svd.tensor} function in the the package \texttt{tensorA} in \texttt{R}; see \cite{boogaart2015}. The eigenfunctions $\varphi_i(t,s)$ may then be estimated with

$$
\hat{\varphi}_i(t,s) = \sum_{j,k=1}^{q^2} \hat{{\bf \Phi}}_i[j,k]\theta_{X,j}(t)\theta_{Y,k}(s).
$$
We use these estimates for the $\hat{\lambda}_i's$ and $\hat{\varphi}_i's$ in the simulations and application below.

\section{Simulation Study } \label{simul}

\subsection{Outline}
We now present the results of a small simulation experiment which aimed to study the finite sample properties of the test statistics introduced above. All of the simulations reported below were done using the \texttt{R} language (\cite{rcore2015}). The number of potential experimental settings that could be considered to study Theorem \ref{th-1} and Corollaries \ref{cor-1}-\ref{cor-3} is enormous. For the sake of brevity, we mainly focused on demonstrating that for a somewhat rich class of data generating processes (DGP's) exhibiting serial correlation, the tests based on $F_T$, $F_{T,p}$, $Z_T$, and $Z_{T,p}$ hold their size well, and that the eigenvalue and eigenfunction estimation procedure explained in Section \ref{num-imp} is adequate for such hypothesis testing problems. Towards this goal, we considered the following basic structure for generating synthetic data depending on the parameter $\alpha \in [0,1]$:

\begin{align}\label{DGP}
X_i(t) = \alpha \varepsilon_{c,i}(t) + (1-\alpha) \varepsilon_{x,i}(t),\;\; Y_i(t) = \alpha \varepsilon_{c,i}(t) + (1-\alpha) \varepsilon_{y,i}(t),\;\; 1\le i \le T
\end{align}
where $\varepsilon_{c,i}(t)$, $\varepsilon_{x,i}(t) $, and $\varepsilon_{y,i}(t)$ are mutually independent sequences that satisfy either of two models given below. In this case, the functional series $X_i(t)$ and $Y_i(t)$ are correlated through their common dependence on $\varepsilon_{c,i}(t)$, with the strength of this dependence controlled by $\alpha$. If $\alpha=0$ for instance, then the two sequences $X_i$ and $Y_i$ are independent. We took the error sequences to satisfy either:

\begin{itemize}
  \item[]{\bf IID:} $\varepsilon_{c,i}(t)=W_{c,i}(t)$, $\varepsilon_{x,i}(t)=W_{x,i}(t) $, and $\varepsilon_{y,i}(t)=W_{y,i}(t)$, where $\{W_{c,i}(t)\}$, $\{W_{x,i}(t)\}$, $\{W_{y,i}(t)\}$ are mutually independent sequences of IID standard Brownian motions.
  \item[]{\bf FAR(1):}  $\varepsilon_{c,i}(t)=\int \Phi(t,s)\varepsilon_{c,i}(s)ds+W_{c,i}(t)$, $\varepsilon_{x,i}(t)=\int \Phi(t,s)\varepsilon_{x,i}(s)ds+W_{x,i}(t)$, and $\varepsilon_{y,i}(t)=\int \Phi(t,s)\varepsilon_{y,i}(s)ds+W_{y,i}(t)$, where $\{W_{c,i}(t)\}$, $\{W_{x,i}(t)\}$, and $\{W_{y,i}(t)\}$ are defined above, and $\Phi(t,s)=\min(t,s)$.
\end{itemize}
The sequence $\{(X_i,Y_i)\}$ satisfying \eqref{DGP} with error sequences satisfying either of the above models satisfy Assumption \ref{edep}. The choice of the Brownian motions for the innovation sequences is partially motivated by our application to intraday returns data below, see Figure \ref{fig-PCIDR}. For each setting of $T$ and $\alpha$, the data was generated on an equally spaced grid on $[0,1]$ with $R=100$ points, and the simulation was repeated 1000 times for each setting in order to calculate the empirical size and power curves presented below. We chose $q$ to be 3 in the estimation of procedure for for the eigenvalues and eigenfunctions in Section \ref{imp}, which in the vernacular of fPCA corresponds to a total variance explained of around 93\% on average for each sequence $X_i$ and $Y_i$. We considered larger values of $q$, and found that it did not have much of an effect on the results, although the computational time increases substantially as $q$ increases. We similarly took $n=3$ in the definition of $F_{T,p}$ and $Z_{T,p}$.

To study the size of the test of $H_{0,1}$ based on $F_T$ and $F_{T,3}$, we generated data according to \eqref{DGP} for each setting of the error sequence with $\alpha=0$, and performed tests with nominal levels of 10\%, 5\%, and 1\% of $H_{0,1}: C_{XY}=0$. The rejection rates from 1000 simulations from each statistic are reported in Table \ref{h01tab}. One thing that was clear based on these simulations was that the tests based on either statistic tended to be somewhat oversized. This issue improves as $T$ increases, as predicted by Corollaries \ref{cor-1} and \ref{cor-2}, and all tests achieved quite good size once $T\ge 300$. The presence of serial correlation pronounces the size inflation of the tests, however this is fairly well controlled by the simple weight function-bandwidth smoothed periodogram estimator. This could be improved further by increasing the bandwidth at the cost of increased computational time, according to unreported simulations. The fact that the empirical levels are quite close to nominal and improve with increasing $T$ is indicative that the eigenvalue/eigenfunction estimation described in Section \ref{num-imp} is performing adequately for this application. We aim in future work to improve the size properties of the test under serial correlation by implementing a data driven bandwidth selection routine, and to consider alternative methods for estimating the eigenvalues/eigenfunctions of $d$.

{\small
\begin{table}[h]
\begin{center}
\begin{tabular}{c ccc c ccc c}
\toprule
 \multicolumn{9}{c}{Statistic: $F_T$} \\
 & \multicolumn{3}{c}{IID} & & \multicolumn{3}{c}{FAR(1)} & \\
\cmidrule{2-4} \cmidrule{6-8}
T & 10\% &5\%  &1\% & & 10\% &5\% & 1\% & \\
\cmidrule{1-1} \cmidrule{2-4} \cmidrule{6-8}
50  & 0.133 & 0.073 & 0.013 & & 0.155 & 0.087 & 0.018 &\\
100 & 0.129 & 0.061 & 0.015 & & 0.148 & 0.069 & 0.020 &\\
300 & 0.114 & 0.049 & 0.008 & & 0.139 & 0.077 & 0.018 &\\
\hline
 \multicolumn{9}{c}{Statistic: $F_{T,3}$} \\
 & \multicolumn{3}{c}{IID} & & \multicolumn{3}{c}{FAR(1)} & \\
\cmidrule{2-4} \cmidrule{6-8}
T & 10\% &5\%  &1\% & & 10\% &5\% & 1\% & \\
\cmidrule{1-1} \cmidrule{2-4} \cmidrule{6-8}
50 &  0.142 & 0.078 & 0.026 & & 0.171 & 0.096 & 0.025 & \\
100 & 0.118 & 0.073 & 0.019 & & 0.189 & 0.087 & 0.034 & \\
300 & 0.121 & 0.041 & 0.012 & & 0.134 & 0.076 & 0.021 & \\
\bottomrule
\end{tabular}
\caption{Empirical sizes with nominal levels of 10\%, 5\%, and 1\% for a test of $H_{0,1}:\ C_{XY} = 0$ where the data was generated according to \eqref{DGP} with $\alpha=0$.}\label{h01tab}
\end{center}
\end{table}
}

We also studied the empirical power of $F_T$ and $F_{T,p}$ by testing $H_{0,1}: C_{XY}=0$ with data following \eqref{DGP} and $\alpha$ increasing from $0$ to $1$. The results of this simulation are reported in the case FAR(1) errors as power curves in Figure \ref{fig-powcur}. We observed that both tests were powerful for large enough values of $\alpha$, although it is difficult to quantify their power without a comparable test. Interestingly, the simple norm based test based on $F_T$ possessed better power than $F_{T,3}$ in all the examples that we considered.

\begin{figure}
        \centering
 \mbox{\subfigure{\includegraphics[width=3.1in]{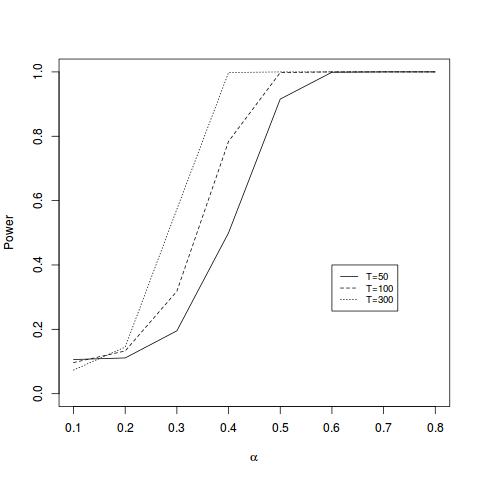}} \subfigure{\includegraphics[width=3.1in]{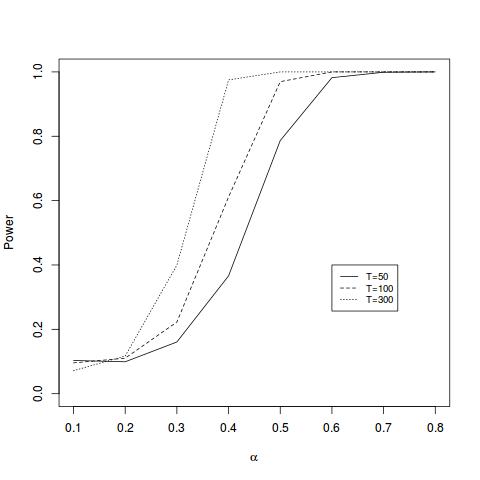} }   } \\
 \caption{Power curves as a function of $\alpha$ ranging from 0 to 0.8 at increments of 0.1, for a test $H_{0,1}:C_{XY}=0$ with level of 5\% applied to $(X_i,Y_i)$ following \eqref{DGP} with FAR(1) errors using $F_T$ (left hand panel), and $F_{T,3}$ (right hand panel).}\label{fig-powcur}
 %\mbox{\subfigure{\includegraphics[width=3.1in]{normIIDpow.jpg}} \subfigure{\includegraphics[width=3.1in]{normARpow.jpg} }   }
\end{figure}

Regarding the test of $H_{0,2}$ based on $Z_T$ and $Z_{T,3}$, we only present the results of an empirical size study. In this case we applied each statistic to data generated according to \eqref{DGP} with $\alpha=0$ and $\alpha=0.5$ for $T=100,300$ and $1000$. The empirical size from 1000 independent simulations for each setting is presented in Table \ref{tab-cp}. We saw that in general the change point tests based on $Z_T$ and $Z_{T,3}$ had quite good size, even in the presence of serial correlation. Changing $\alpha$ had no effect on the size of these tests, as expected.

{\small
\begin{table}[h]
\begin{center}
\begin{tabular}{c c ccc c ccc c}
\toprule
\multicolumn{9}{c}{Statistic: $Z_{T}$} \\
 & & \multicolumn{3}{c}{IID} & & \multicolumn{3}{c}{FAR(1)} & \\
\cmidrule{3-5} \cmidrule{7-9}
T & $\alpha$ & 10\% &5\%  &1\% & & 10\% &5\% & 1\% & \\
\cmidrule{1-1} \cmidrule{2-2} \cmidrule{3-5} \cmidrule{7-9}
100	 & 0	 & 0.093 & 0.048 & 0.007 & & 0.149 & 0.087 & 0.031 & \\
	   & 0.5 & 0.125 & 0.066 & 0.020 & & 0.094 & 0.046	& 0.006 & \\
300	 & 0	 & 0.102 & 0.047 & 0.008 & & 0.122 & 0.079 & 0.025 & \\
		 & 0.5 & 0.115 & 0.060 & 0.020 & & 0.121 & 0.057 & 0.015 & \\
1000 & 0   & 0.133	& 0.071 & 0.006 & & 0.146 & 0.063 & 0.010 & \\
		 & 0.5 & 0.123	& 0.058 & 0.012 & & 0.139 & 0.075 & 0.012 & \\
\hline
\multicolumn{9}{c}{Statistic: $Z_{T,3}$} \\
 & & \multicolumn{3}{c}{IID} & & \multicolumn{3}{c}{FAR(1)} & \\
\cmidrule{3-5} \cmidrule{7-9}
T & $\alpha$ & 10\% &5\%  &1\% & & 10\% &5\% & 1\% & \\
\cmidrule{1-1} \cmidrule{2-2} \cmidrule{3-5} \cmidrule{7-9}
100	 & 0	  & 0.108 & 0.060 & 0.016 & & 0.166 & 0.098 & 0.027 & \\
		 & 0.5	& 0.114 & 0.070 & 0.025 & & 0.132 & 0.058 & 0.016 & \\
300	 & 0	  & 0.132 & 0.073 & 0.021 & & 0.132 & 0.067 & 0.011 & \\
	   & 0.5	& 0.121 & 0.056 & 0.015 & & 0.104 & 0.062 & 0.013& \\
1000 & 0	  & 0.143 & 0.082 & 0.024 & & 0.102 & 0.058 & 0.012 & \\
	   & 0.5	& 0.126 & 0.067 & 0.014 & & 0.094 & 0.052 & 0.010 & \\
\bottomrule
\end{tabular}
\caption{ Empirical sizes with nominal levels of 10\%, 5\%, and 1\% for a test of $H_{0,2}:\ C_{XY}^{(1)} = \hdots = C_{XY}^{(T)}$ where the data was generated according to \eqref{DGP} with $\alpha=0$ and $\alpha=0.5$. }\label{tab-cp}
\end{center}
\end{table}
}

\section{Application to cumulative intraday returns}\label{app}

A natural example of functional time series data are those derived from densely recorded asset price data, such as intraday stock price data. Recently there has been an upsurge in quantitative research focused on analyzing the information contained within curves constructed from such data; we refer the reader to \cite{barndorff:2004}, \cite{wang:zou:2010}, \cite{gabrys:horvath:kokoszka:2010}, \cite{muller:sen:stadtmuller:2011}, and \cite{kokoszka:reimherr:2013}. Price curves associated with popular companies are commonly displayed at websites like yahoo.com/finance, and the ``patterns" in them, or lack thereof, are of apparent interest to day traders.

The specific data that we consider was obtained from www.nasdaq.com, and consists of 1 minute resolution closing prices of a single share of Microsoft (ticker \texttt{MSFT}), and Exxon Mobile (ticker \texttt{XOM}) stock from January 2nd to December 31st, 2001, which comprises data from 248 trading days ($T=248$) with $R=389$ observations per day. We aimed to apply the methods introduced in Sections \ref{i-1} to study the cross-covariance structure between these two companies stock prices on the intraday scale. Let $P_{M,i}(t_j),$ and $P_{X,i}(t_j)$ $i=1,\ldots, T, j=1,\ldots, R$, denote the prices of Microsoft and Exxon mobile stock on day $i$ at intraday time $t_j$, respectively. The first three price curves of each series constructed from the raw price data and linear interpolation are displayed in the left hand panel of Figure \ref{fig-PCIDR}. The functional time series of price curves are evidently nonstationary due to frequent level shifts and volatility, and hence we considered the following transformation of these curves akin to taking the log returns for scalar price data:

\begin{defi} \label{def:icr}
Suppose $P_{i}(t_j), i=1,\ldots, T, j=1,\ldots, R$, is the price of
a financial asset at time $t_j$ on day $i$. The functions
\[
r_{i}(t_j) = 100[\ln P_{i}(t_j)-\ln P_{i}(t_1)], \ \ \ \
j = 1, 2,\ldots, R,\ \ \ \
i =1, \ldots, T,
\]
are called the {\em  cumulative intraday returns} (CIDR's).
\end{defi}

Since the logarithm is increasing, the CIDR curves have nearly the same shape as the original daily price curves, but the assumption of stationarity is much more plausible for these curves. According to their definition, the CIDR's always start from zero, so level stationarity is enforced, and taking the logarithm helps reduce potential scale inflation. The stationarity of CIDR curves derived from intraday stock price data was argued empirically in \cite{horvath:kokoszka:rice:2014}.

Let $X_i(t)=r_{M,i}(t)$ and $Y_i(t)=r_{X,i}(t)$, $1 \le i \le T$, denote the CIDR curves derived from the Microsoft and Exxon Mobile stock price data and linear interpolation, respectively. The first three CIDR curves of each series are plotted in the right hand panel of Figure \ref{fig-PCIDR}.

\begin{figure}
        \centering
 \mbox{\subfigure{\includegraphics[width=3.1in]{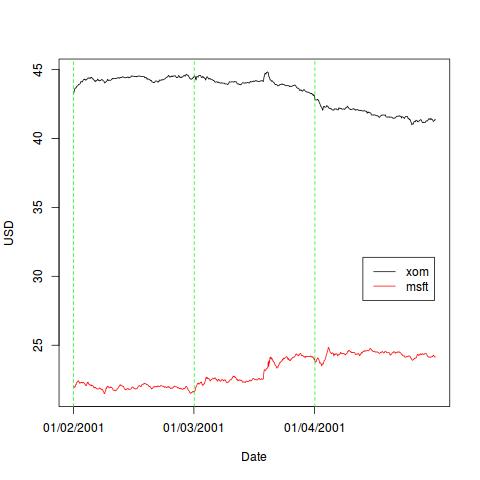}} \subfigure{\includegraphics[width=3.1in]{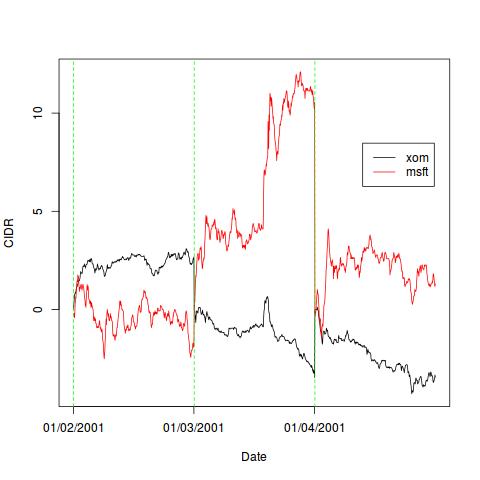} }   }
        \caption{The left hand panel displays three intraday price curves derived from the one-minute resolution closing prices of a single share of \texttt{XOM} and \texttt{MSFT}. The right hand panel displays the corresponding CIDR curves.}\label{fig-PCIDR}

\end{figure}

Before estimating and measuring the significance of the cross-covariance between these functional time series of CIDR's, we first tested for its homogeneity within the sample, a test of $H_{0,2}$, using the test statistic $Z_T$. For this analysis, we took $q=3$ for the method outline in Section \ref{num-imp}, which corresponded to approximately 94\% variance explained in each series. A ``CUSUM chart" of $c(t)=T \| \hat{C}_{XY}(\cdot,\cdot,t/T)- (t/T)\hat{C}_{XY}(\cdot,\cdot,1) \|_2^2$ versus $t$ is given in Figure \ref{fig-cusum} with the corresponding 10\%, 5\% and 1\% significance levels of the estimated limiting distribution in Corollary \ref{cor-1}. The statistic $Z_T$ far exceeded the 1\% level, which indicates that there is strong evidence that the covariance relationship is heterogenous within the sample. Also apparent in the plot, the largest difference between the partial sample cross-covariance estimate occurs on March 22nd, 2001. We segmented the data into two sub-samples before and after this point of lengths $T_1=56$ and $T_2=192$, respectively, and again tested for the homogeneity of the cross-covariance within each sub-sample. In both sub-samples the homogeneity could not be rejected with any significance. Interestingly, the sample including the intraday  \texttt{XOM} and \texttt{MSFT} returns data before and after the terrorist attacks on September 11, 2001 did not seem to exhibit a change point in the cross-covariance relationship. The date March 22nd, 2001 does however seem to be in the proximity of some fairly important events relative to the world oil market. On March 17th, OPEC had announced that they were cutting oil production by 4\%, and just three days later the largest off shore oil rig in the world, Petrobras 36, sank off of the coast of Brazil.

In order to measure the significance of the estimates of the cross-covariance before and after the change point, we applied a test of $H_{0,1}:C_{XY}=0$ to each sub-sample. The null hypothesis was strongly rejected in both cases with $p$-values smaller than 0.000. The cross-covariance surfaces are displayed in Figure \ref{fig-covs}, from which we can see that the shape of the cross-covariance is quite different before and after the initial change point: the surface goes from being negative to predominantly positive.

\begin{figure}
        \centering
         \mbox{\subfigure{\includegraphics[width=3.1in]{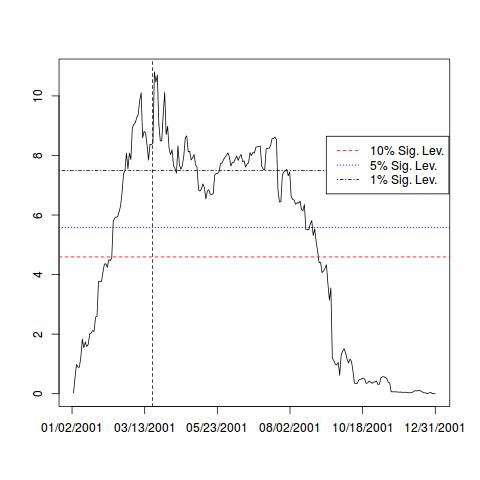}}   }
 \mbox{\subfigure{\includegraphics[width=3.1in]{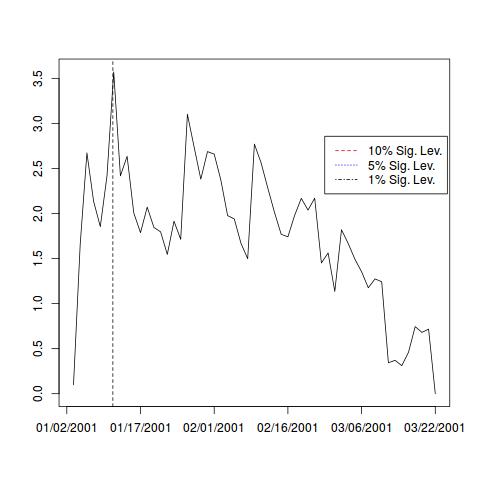}} \subfigure{\includegraphics[width=3.1in]{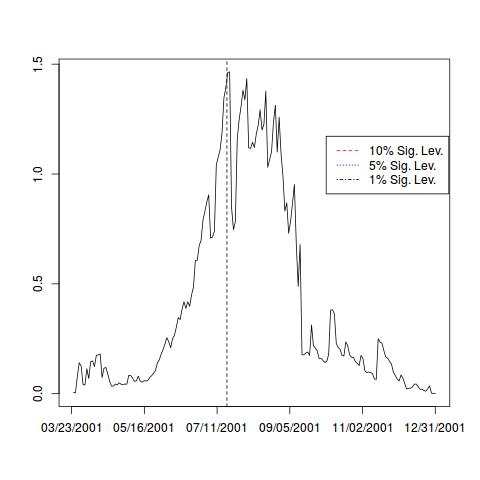} }   }
        \caption{ The top panel displays the ``CUSUM chart" of  $\| \hat{C}_{XY}(\cdot,\cdot,t/T)- (t/T)\hat{C}_{XY}(\cdot,\cdot,1) \|_2^2$ versus $t$ calculated from the \texttt{XOM} and \texttt{MSFT} price data from 2001. The maximum value $Z_T$ is significant to the 0.01 level, and is achieved at the value $t$ corresponding to March 22nd. The lower left and lower right panels show similar CUSUM charts for the sub samples of data before and after this point, respectively.}\label{fig-cusum}

\end{figure}

\begin{figure}
        \centering
 \mbox{\subfigure{\includegraphics[width=3.1in]{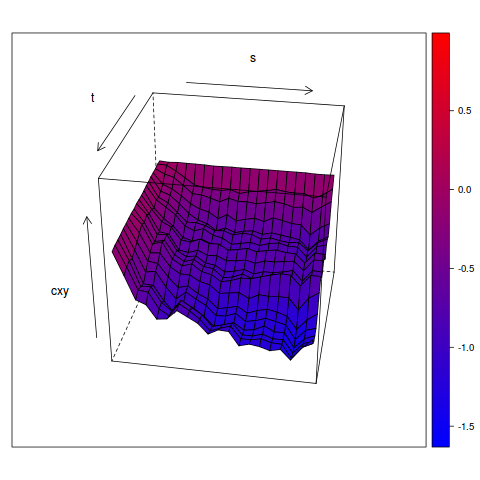}} \subfigure{\includegraphics[width=3.1in]{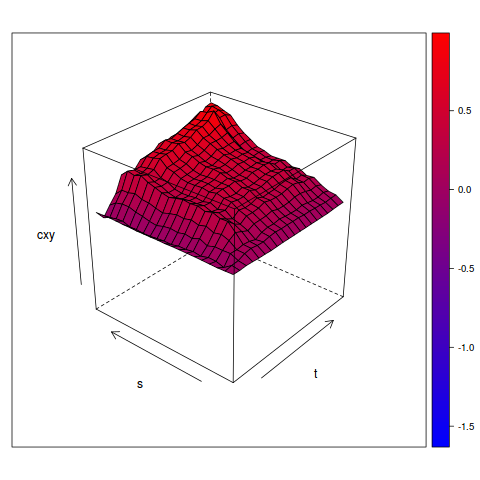} }   }
        \caption{The left and right hand panels display the estimator $\hat{C}_{XY}$ before and after the date 3/22/2001, respectively. Both of these surface estimates were measured to be significantly different from zero, as the test of $H_{0,1}: \; C_{XY}=0$  based on $F_T$  yielded $p$-values less than 0.000 in both cases. }\label{fig-covs}

\end{figure}

\pagebreak

\bibliographystyle{plain}

\bibliography{glp}

\appendix

\section{Proofs of technical results} \label{proofs}

\subsection{Proof of Theorem \ref{th-1}}

In order to prove Theorem \ref{th-1}, we may assume without loss of generality that $\mu_X=\mu_Y=0$. Let

$$
\tilde{C}_{XY}(t,s,x) = \frac{1}{T} \sum_{i=1}^{\lfloor Tx \rfloor} X_i(t)Y_i(s).
$$

\begin{lemma}\label{lem-1}
Under Assumption \ref{edep},

$$
\sup_{0\le x \le 1 } \|\hat{C}_{XY}(\cdot,\cdot ,x) -\tilde{C}_{XY}(\cdot,\cdot ,x)\|_2 = O_P\left(\frac{1}{T}\right).
$$

\begin{proof}
According to the definitions of $\hat{C}_{XY}$ and $\tilde{C}_{XY}$ and the triangle inequality,

\begin{align}\label{approx-1}
 \|\hat{C}_{XY}(\cdot,\cdot ,x) -\tilde{C}_{XY}(\cdot,\cdot ,x)\|_2 &= \frac{1}{T} \Biggl\| \sum_{i=1}^{\Tx} (X_i - \bar{X})\otimes (Y_i- \bar{Y}) - X_i \otimes Y_i   \Biggl\|_2 \\
 &= \frac{1}{T}  \Biggl\| \sum_{i=1}^{\Tx} ( - \bar{X} \otimes Y_i- X_i \otimes \bar{Y} + \bar{X} \otimes \bar{Y} )   \Biggl\|_2 \\
 &\le \frac{1}{T}  \Biggl\| \sum_{i=1}^{\Tx} \bar{X} \otimes Y_i \Biggl\|_2 + \frac{1}{T} \Biggl\|\sum_{i=1}^{\Tx} X_i \otimes \bar{Y}\Biggl\|_2 + \frac{1}{T} \Biggl\|\sum_{i=1}^{\Tx} \bar{X} \otimes \bar{Y}   \Biggl\|_2 \\
&= G_1(x)+G_2(x)+G_3(x).
\end{align}

For the term $G_1(x)$, we have that

\begin{align}\label{g1-eq}
G_1(x) = \frac{1}{T} \Biggl\| \sum_{i=1}^{\Tx} \bar{X} \otimes Y_i \Biggl\|_2= \frac{1}{T}  \|\bar{X}\|_1\Biggl\| \sum_{i=1}^{\Tx}Y_i \Biggl\|_1.
\end{align}

Assumption \ref{edep} implies that both the series $X_i$ and $Y_i$ satisfy the conditions of Theorem 3.3 of Berkes et al \cite{berkes:horvath:rice:2013}, from which it follows that

$$
\|\bar{X}\|_1=O_P\left(\frac{1}{\sqrt{T}}\right),\mbox{ and }\sup_{0 \le x \le 1} \Biggl\| \sum_{i=1}^{\Tx}Y_i \Biggl\|_1=O_P\left(\sqrt{T}\right).
$$
This combined with \eqref{g1-eq} implies that $\sup_{0\le x \le 1}G_1(x)=O_P(1/T)$. Parrel arguments show that $\sup_{0\le x \le 1}G_2(x)=O_P(1/T)$ and $\sup_{0\le x \le 1}G_3(x)=O_P(1/T)$, from which the result follows in light of \eqref{approx-1}.
\end{proof}
\end{lemma}

\begin{lemma}\label{lem-2} Under Assumption \ref{edep} and if $q=p/2$ with $p$ defined in Assumption \ref{edep}, then

$$
(E\|X_i \otimes Y_i - X_{i,m} \otimes Y_{i,m}\|_2^q )^{1/q}=O(m^{-\gamma}),
$$
where $\gamma=\min(\alpha,\beta)$, and $X_{i,m}$ and $Y_{i,m}$ are defined in Assumption \ref{edep}.

\begin{proof}

We have according to the triangle inequalities in $L^2[0,1]^2$ and for $(E(\cdot)^q)^{1/q}$ that

\begin{align}\label{lem21}
(E\| X_i \otimes  Y_i - X_{i,m} \otimes Y_{i,m} \|_2^q )^{1/q} &= (E\| X_i \otimes Y_i - X_i \otimes Y_{i,m}+ X_{i,m} \otimes Y_i -  X_{i,m} \otimes Y_{i,m} \|_2^q )^{1/q} \\
&\le (E (\| X_i \otimes Y_i - X_i \otimes Y_{i,m} \|_2 + \|X_{i,m} \otimes Y_i -  X_{i,m} \otimes Y_{i,m} \|_2)^q )^{1/q} \notag \\
&\le (E \| X_i \otimes Y_i - X_i \otimes Y_{i,m} \|_2^q)^{1/q} + (E\|X_{i,m} \otimes Y_i -  X_{i,m} \otimes Y_{i,m} \|_2^q )^{1/q} \notag \\
&=   (E \| X_i \|_1^q \| Y_i - Y_{i,m} \|_1^q)^{1/q} + (E \| Y_i \|_1^q \| X_i - X_{i,m} \|_1^q)^{1/q} \notag.
\end{align}
According to the Cauchy-Schwarz inequality and stationarity, we have that

\begin{align}\label{lem22}
E \| X_i \|_1^q \| Y_i - Y_{i,m} \|_1^q \le (E \| X_i \|_1^{2q})^{1/2} E(\| Y_i - Y_{i,m} \|_1^{2q})^{1/2} = (E \| X_0 \|_1^{2q})^{1/2} E(\| Y_0 - Y_{0,m} \|_1^{2q})^{1/2}.
\end{align}
One obtains a similar bound with the roles of $X$ and $Y$ swapped, which combined with the last line of \eqref{lem21} we get that (with $p=2q$)

$$
(E\| X_i \otimes Y_i - X_{i,m} \otimes Y_{i,m} \|_2^q )^{1/q} \le (E \| X_0 \|_1^{p})^{1/p} (E\| Y_0 - Y_{0,m} \|_1^{p})^{1/p}+ (E \| Y_0 \|_1^{p})^{1/p} (E\| X_0 - X_{0,m} \|_1^{p})^{1/p} .
$$
Since both $(E \| X_0 \|_1^{p})^{1/p}$ and $(E \| Y_0 \|_1^{p})^{1/p}$ are finite, and according to Assumption \ref{edep} $(E\|X_i-X_{i,m}\|^p\big)^{1/p}=O(m^{-\alpha}),\mbox{ and } (E\|Y_i-Y_{i,m}\|^p\big)^{1/p}=O(m^{-\beta})$, we obtain that

$$
(E\|X_i \otimes Y_i - X_{i,m} \otimes Y_{i,m}\|_2^q )^{1/q}=O(m^{-\gamma}),
$$
where $\gamma=\min(\alpha,\beta)$ as needed.

\end{proof}

\end{lemma}

\begin{proof}[Proof of Theorem \ref{th-1}]
According to Lemma \ref{lem-1}, it is enough to prove Theorem \ref{th-1} for the process $\tilde{C}_{XY}(t,s,x)$. $\tilde{C}_{XY}(t,s,x)-(\Tx / T) C_{XY}(t,s)$ is the partial sum of the random functions $\xi_i(t,s)= X_i(t)Y_i(s)-EX_0(t)Y_0(s)$ which form a mean zero and stationary sequence in $L^2[0,1]^2$. Let $\xi_{i,m}(t,s)= X_{i,m}(t)Y_{i,m}(s)-EX_0(t)Y_0(s)$, denote an $m$-dependent approximation to $\xi_i$. It follows directly from $\ref{lem-2}$ that
$$
(E \| \xi_i - \xi_{i,m} \|_2^q)^{1/q} = O(m^{-\gamma}),
$$
where $\gamma>1$, and this implies that

$$
\sum_{m=1}^\infty (E \| \xi_i - \xi_{i,m} \|_2^q)^{1/q} < \infty.
$$

It now follows from Theorem 1 of \cite{jirak:2013} that if

$$
 \Xi_T(t,s,x) :=\sqrt{T}\left(\tilde{C}_{XY}(t,s,x)-\frac{\Tx }{ T} C_{XY}(t,s) \right)= \frac{1}{\sqrt{T}} \sum_{i=1}^{\Tx} \xi_i(t,s),
$$
then

\begin{align}\label{th-1-eq-a}
\sup_{0\le x \le 1} \intt \left( \Xi_T(t,s,x) -\Gamma_T(t,s,x)\right)^2dtds = o_P(1),
\end{align}

where $\Gamma_T$ has the properties attributed in Theorem \ref{th-1}, which establishes the theorem.

\end{proof}

To see how Corollaries \ref{cor-1}-\ref{cor-3} follow from this result, we note that for all $T$,

$$
\{\Gamma_T(t,s,x) : \; t,s,x\in[0,1]\} \stackrel{D}{=} \{\Gamma_0(t,s,x) : \; t,s,x\in[0,1]\},
$$
where

$$
\Gamma_0(t,s,x)= \sum_{\ell=1}^{\infty } \lambda_\ell^{1/2} W_\ell(x)\varphi_\ell(t,s),
$$
with $\{W_\ell(t),\; t\in[0,1],\; \ell \in \mathbb{N}\}$ being iid standard Brownian motions, and $(\lambda_\ell$ $\varphi_\ell)$ being the eigenvalues and eigenfunctions of the operator $d$. This follows from simply calculating the mean and covariance function of $\Gamma_0$ and using Mercer's theorem.

\begin{proof}[Proof of Theorem 3.1]
  The first part of the theorem follows directly from the ergodic theorem in Hilbert spaces; see for example Appendix A of \cite{horvath:huskova:rice:2013}. In order to get the second result, we have using the assumption that $\sqrt{T}(C_{XY} - C_0) \stackrel{L^2[0,1]^2}{\to} C_A$ and Theorem \ref{th-1} that, in case of $F_T$,

\begin{align*}
  F_T &= T \| \hat{C}_{XY}(\cdot,\cdot,1) - C_{XY}+C_{XY} - C_0 \|_2^2 \\
   &=  T \langle (\hat{C}_{XY}(\cdot,\cdot,1) - C_{XY})+(C_{XY} - C_0), (\hat{C}_{XY}(\cdot,\cdot,1) - C_{XY})+(C_{XY} - C_0) \rangle_2 \\
   &=  T \| \hat{C}_{XY}(\cdot,\cdot,1) - C_{XY}\|_2^2+ 2\langle \sqrt{T}(\hat{C}_{XY}(\cdot,\cdot,1) - C_{XY}) ,\sqrt{T}(C_{XY} - C_0) \rangle_2 +T\|C_{XY} - C_0\|_2^2\\
   &\stackrel{D}{\to}  \sum_{i=1}^\infty \lambda_i \mN_i^2 + 2 \sum_{i=1}^\infty \lambda^{1/2}_i \langle C_A, \varphi_i \rangle \mN_i + \|C_A\|_2^2 , \;\; (T\to \infty),
\end{align*}
  as needed. The result for $F_{T,p}$ follows similarly.
\end{proof}

\subsection{Proof of Theorem \ref{th-est}}

Let

\begin{align}\label{est-2}
\tilde{D}_T(t,s,u,v)=\sum_{\ell=-\infty}^{\infty}W_{b} \left( \frac{\ell}{h} \right) \tilde{\gamma}_\ell(t,s,u,v),
\end{align}

with $h$ defined in \eqref{h-cond}, $W_b$ being a bounded, symmetric, and continuous weight function of order $b$, and

\begin{align*}
   \tilde{\gamma}_\ell(t,s,u,v)=\left\{
     \begin{array}{lr}
      \displaystyle \frac{1}{T}\sum_{j=1}^{T-\ell}\left( X_j(t)Y_j(s)-EX_0(t)Y_0(s) \right)\left( X_{j+\ell}(t)Y_{j+\ell}(s)-EX_0(t)Y_0(s)\right),\quad &\ell \ge 0
      \vspace{.3cm} \\
     \displaystyle \frac{1}{T}\sum_{j=1-\ell}^{T}\left( X_j(t)Y_j(s)-EX_0(t)Y_0(s) \right)\left( X_{j+\ell}(t)Y_{j+\ell}(s)-EX_0(t)Y_0(s)\right),\quad &\ell < 0.
     \end{array}
   \right.
\end{align*}

\begin{lemma}\label{est-lem-1} Under Assumption \ref{edep},
$$
\|\hat{D} - \tilde{D} \|_4 = o_P(1).
$$
\end{lemma}

\begin{proof}
According to the definition of $\hat{\gamma}_\ell$, when $\ell \ge 0$, we have that

\begin{align}\label{approx-eq-0}
   \hat{\gamma}_\ell&(t,s,u,v)= \frac{1}{T}\sum_{j=1}^{T-\ell}( \bar{X}_j(t) \bar{Y}_j(s)-\hat{C}_{XY}(t,s,1))( \bar{X}_{j+\ell}(u)\bar{Y}_{j+\ell}(v)-\hat{C}_{XY}(u,v,1)) \\
    &= \frac{1}{T}\sum_{j=1}^{T-\ell}\Biggl(  (X_j(t)Y_j(s)-EX_0(t)Y_0(s)) - \bar{X}(t)Y_j(s) + \bar{X}(t)\bar{Y}(s) - X_j(t)\bar{Y}(s) \notag \\
    &\; + (EX_0(t)X_0(s) - \hat{C}(t,s,1)) \Biggl) \times \Biggl(  X_{j+\ell}(u)Y_{j+\ell}(v)-EX_0(u)Y_0(v) - \bar{X}(u)Y_{j+\ell}(v) \notag \\
    &+ \bar{X}(u)\bar{Y}(v) - X_{j+\ell}(u)\bar{Y}(v) + (EX_0(u)X_0(v) - \hat{C}(u,v,1)) \Biggl) \notag \\
    &= \tilde{\gamma}(t,s,u,v) + \sum_{p=1}^{24} R_{p,\ell}(t,s,u,v), \notag
\end{align}
where the terms $R_{p,\ell}$ represent the remaining 24 terms in the definition of $\hat{\gamma}_\ell$ obtained by completing the multiplication that results in 25 terms, the first of which corresponds to $\tilde{\gamma}_\ell$. We now aim to show that for all $p$ and $\ell$, $E\|R_{p,\ell}\|_4=O(T^{-1/2})$. For $R_{1,\ell}$, we have by the triangle inequality, a couple of applications of the Cauchy-Schwarz inequality, and the assumed stationarity that

\begin{align}\label{approx-eq-1}
E\|R_{1,\ell}\|_4 &= \frac{1}{T} E \Biggl\| \sum_{j=1}^{T-\ell}  (X_j(t)Y_j(s)-EX_0(t)Y_0(s))\bar{X}(u)Y_{j+\ell}(v) \Biggl\|_4 \\
&\le \frac{1}{T}  \sum_{j=1}^{T-\ell}  E \|(X_j(t)Y_j(s)-EX_0(t)Y_0(s))\bar{X}(u)Y_{j+\ell}(v) \|_4 \notag \\
&= \frac{1}{T}  \sum_{j=1}^{T-\ell}  E  \|(X_j(t)Y_j(s)-EX_0(t)Y_0(s))\|_2 \|\bar{X}\|_1 \|Y_{j+\ell}\|_1 \notag \\
&\le \frac{1}{T}  \sum_{j=1}^{T-\ell}   (E  \|(X_j(t)Y_j(s)-EX_0(t)Y_0(s))\|_2^2)^{1/2} (E (\|\bar{X}\|_1 \|Y_{j+\ell}\|_1)^2)^{1/2} \notag \\
&\le \frac{1}{T}  \sum_{j=1}^{T-\ell}  (E  \|(X_j(t)Y_j(s)-EX_0(t)Y_0(s))\|_2^2)^{1/2} (E \|\bar{X}\|_1^4)^{1/4} (E \|Y_{j+\ell}\|_1^4)^{1/4} \notag\\
&= \frac{T-\ell}{T}   (E  \|(X_0(t)Y_0(s)-EX_0(t)Y_0(s))\|_2^2)^{1/2} (E \|\bar{X}\|_1^4)^{1/4} (E \|Y_0\|_1^4)^{1/4}. \notag
\end{align}

It follows from Proposition 3.1 of \cite{torgovitski:20162} that under Assumption \ref{edep}, $(E \|\bar{X}\|_1^4)^{1/4}=O(T^{-1/2})$, and also according to Assumption \ref{edep}, $(E \|Y_0\|_1^4)^{1/4}=O(1)$ and $(E  \|(X_0(t)Y_0(s)-EX_0(t)Y_0(s))\|_2^2)^{1/2}=O(1)$. This implies with \eqref{approx-eq-1} that $E\|R_{1,\ell}\|_4=O(T^{-1/2})$. Similar arguments using the fact that $(E \|\bar{Y}\|_1^4)^{1/4}=O(T^{-1/2})$ and $\| EX_0(t)Y_0(s) - \hat{C}_{XY}(t,s,1)\|_2=O(T^{-1/2})$ can be applied to get that $E\|R_{p,\ell}\|=O(T^{-1/2})$, $2\le p \le 24$. This implies via the triangle inequality and \eqref{approx-eq-0} that

$$
E\| \hat{\gamma}_\ell - \tilde{\gamma}_\ell\|_4 = O(T^{-1/2}),
$$
from which it follows that

\begin{align}\label{approx-eq-2}
E\Biggl\| \sum_{\ell=1}^\infty W_b\Biggl(\frac{\ell}{h} \Biggl) (\hat{\gamma}_\ell - \tilde{\gamma}_\ell) \Biggl\|_4 \le \sum_{\ell=1}^\infty W_b\Biggl(\frac{\ell}{h} \Biggl)  E\|(\hat{\gamma}_\ell - \tilde{\gamma}_\ell) \|_4=O(hT^{-1/2})=o(1),
\end{align}
by \eqref{h-cond} and the fact that $W_b$ has bounded support. A parallel argument may be used to show that for $\ell<0$, that $E\| \hat{\gamma}_\ell - \tilde{\gamma}_\ell\|_4 = O(T^{-1/2}),$ from which we obtain that

\begin{align}\label{approx-eq-3}
E\Biggl\| \sum_{\ell=-\infty}^{-1} W_b\Biggl(\frac{\ell}{h} \Biggl) (\hat{\gamma}_\ell - \tilde{\gamma}_\ell) \Biggl\|_4 =O(hT^{-1/2})=o(1).
\end{align}

\eqref{approx-eq-2} and \eqref{approx-eq-3} together with the triangle inequality imply that

$$
E \| \hat{D}- \tilde{D} \|_4 \le E\Biggl\| \sum_{\ell=1}^\infty W_b\Biggl(\frac{\ell}{h} \Biggl) (\hat{\gamma}_\ell - \tilde{\gamma}_\ell) \Biggl\|_4 + E\Biggl\| \sum_{\ell=-\infty}^{-1} W_b\Biggl(\frac{\ell}{h} \Biggl) (\hat{\gamma}_\ell - \tilde{\gamma}_\ell) \Biggl\|_4 = o(1),
$$
and this implies the result with Markov's inequality.
\end{proof}

\begin{lemma}\label{est-lem-2} Under Assumption \ref{edep},
$$
\|\tilde{D} - D \|_4 = o_P(1).
$$
\end{lemma}

\begin{proof}
$\tilde{\gamma}_\ell$ is the autocovariance estimator based on a sample of size $T$ at lag $\ell$ of the mean zero random functions $\{\xi_i(t,s), \; i \in {\mathbb Z}\}$, where, recalling from above, $\xi_i(t,s) = X_i(t)Y_i(s) - EX_0(t)Y_0(s)$. With $\xi_{i,m}(t,s) = X_{i,m}(t)Y_{i,m}(s) - EX_0(t)Y_0(s)$, we have from \ref{lem-2} and Lyapounov's inequality that

$$
\lim_{m\to \infty} m (E\| \xi_i - \xi_{i,m}\|_2^2)^{1/2} = 0,
$$
and
$$
\sum_{m=1}^{\infty} (E\| \xi_i - \xi_{i,m}\|_2^2)^{1/2}<\infty.
$$
This shows that $\xi_i$ satisfy the conditions of Theorem 2 from \cite{horvath:kokoszka:reeder:2012}, which implies the result.

\end{proof}

\begin{proof}[Proof of Theorem \ref{th-est}]
The Theorem follows directly from Lemmas \ref{est-lem-1} and \ref{est-lem-2}.
\end{proof}

\subsection{Consistency of $Z_T$ and $Z_{T,p}$}\label{alt-sec}
In order for the tests based on $Z_T$ and $Z_{T,p}$ to be consistent, we only require that the series before and after the change is ergodic, which we quantify by the following assumption.

\begin{assumption}\label{alt-as-1} Under $H_{A,2}$, we assume that

$$
\Biggl\| \frac{1}{k^*} \sum_{i=1}^{k^*} (X_i(t) - \mu_X(t))(Y_i(s)- \mu_Y(s) ) - C_1(t,s) \Biggl\|_2 = o_P(1),
$$
and
$$
 \Biggl\| \frac{1}{T-k^*} \sum_{i=k^*+1}^{T} (X_i(t) - \mu_X(t))(Y_i(s)- \mu_Y(s) ) - C_2(t,s) \Biggl\|_2=o_P(1).
$$
\end{assumption}

\begin{theorem}\label{alt-cp-1} Under $H_{A,2}$ and Assumption \ref{alt-as-1} holds, then
$$
Z_T \stackrel{P}{\longrightarrow} \infty.
$$
If in addition $C_1$ and $C_2$ in $H_{A,2}$ satisfy $\langle C_1 - C_2 , \varphi_i\rangle_2 \ne 0$ for some $1\le i \le p$, then
$$
Z_{T,p} \stackrel{P}{\longrightarrow} \infty.
$$
\end{theorem}

Assumption \ref{alt-as-1} implies that the covariance estimators of the data before and after the change are consistent with their underlying population quantities. This would hold if, for example, the function $g_{XY}$ which defines how the series $(X_i,Y_i)$ depends on the underlying innovation sequence were to change from $g_{XY}^{(1)}$ to $g_{XY}^{(2)}$ in such a way that $H_{A,2}$ held, since then the series before and after the change would still be ergodic. The proof of Theorem \ref{alt-cp-1} is straightforward, so the details are omitted.

\end{document}